\documentclass[11pt,spanish,english]{amsart}

\usepackage{amsmath, amssymb, amsthm}
\usepackage{geometry}
\usepackage{babel}
\usepackage[unicode=true,
bookmarks=true,bookmarksnumbered=true,bookmarksopen=true,bookmarksopenlevel=3,
breaklinks=false,pdfborder={0 0 0}]{hyperref}                  %backref=page,colorlinks=false]
\hypersetup{pdftitle={On the torsion units of the integral group ring of finite projective special linear groups},
	pdfauthor={Angel del Rio and Mariano Serrano}}

\DeclareMathOperator{\Tr}{Tr}
\DeclareMathOperator{\PSL}{PSL}
\DeclareMathOperator{\SL}{SL}
\DeclareMathOperator{\TPA}{TPA}
\DeclareMathOperator{\VPA}{VPA}
\DeclareMathOperator{\PA}{PA}
\DeclareMathOperator{\diag}{diag}
\DeclareMathOperator{\Odd}{Odd}
\newcommand{\pa}{\Upsilon}

\newtheorem{theorem}{Theorem}[section]
\newtheorem{lemma}[theorem]{Lemma}
\newtheorem{proposition}[theorem]{Proposition}

\newtheorem{corollary}[theorem]{Corollary}
\newtheorem{remark}[theorem]{Remark}

\newcommand{\Z}{{\mathbb Z}}
\newcommand{\Q}{{\mathbb Q}}

\renewcommand{\C}{{\mathbb C}}

\newcommand{\matriz}[1]{\begin{array} #1 \end{array}}
\newcommand{\pmatriz}[1]{\left(\begin{array} #1 \end{array}\right)}
\newcommand{\GEN}[1]{\langle #1 \rangle}

\thanks{The first author has been partially supported by CAPES (Proc nº BEX4147/13-8) of Brazil. The second author has been partially supported by Ministerio de Econom\'{\i}a y Competitividad project MTM2012-35240 and Fondos FEDER and Proyecto Hispano-Brasile\~{n}o de Cooperaci\'{o}n Interuniversitaria PHB-2012-0135.}

\author{\'{A}ngel del R\'{\i}o}
\address{\'{A}ngel del R\'{\i}o, Departamento de Matem\'{a}ticas, Universidad de Murcia,
30100, Murcia, Spain}
\email{adelrio@um.es}

\author{Mariano Serrano}
\address{Mariano Serrano, Department of Mathematics, University of Murcia.}
\email{mariano.serrano@um.es}

\thanks{Partially supported by Ministerio de Econom\'{\i}a y Competitividad project MTM2012-35240 and Fondos FEDER and Fundaci\'{o}n S\'{e}neca of Murcia 19880/GERM/15.}

\title{On the torsion units of the integral group ring of finite projective special linear groups}

\begin{document}

	\begin{abstract}
		H. J. Zassenhaus conjectured that any unit of finite order and augmentation one in the integral group ring of a finite group $G$ is conjugate in the rational group algebra to an element of $G$.
		One way to verify this is showing that such unit has the same distribution of partial augmentations as an element of $G$ and the HeLP Method provides a tool to do that in some cases.
		In this paper we use the HeLP Method to describe the partial augmentations of a hypothetical counterexample to the conjecture for the projective special linear groups.
	\end{abstract}
	
	\maketitle
	
	\section{Introduction}
	
	Let $G$ be a finite group and let $\mathbb{Z}G$ be the integral group ring of $G$.
	Denote by $V(\mathbb{Z}G)$ the group of normalized units (i.e. units of augmentation $1$) in $\mathbb{Z}G$.
	Hans Zassenhaus stated a list of conjectures about finite subgroups of $V(\Z G)$.
	In this paper we deal with the unique one which is still open.
	It states that every torsion element of $V(\Z G)$ is rationally conjugate (i.e. conjugate in the units of $\Q G$) to an element of $G$.
	We refer to this statement as the Zassenhaus Conjecture.
	It has been verified for nilpotent groups \cite{Weiss1991}, for cyclic-by-abelian groups \cite{CaicedoMargolisdelRio2013} and for $p$-group-by-abelian $p'$-groups \cite{HertweckAlgColloq}.
	For non-solvable groups the Zassenhaus Conjecture has been proved only for a few simple or almost simple groups
	\cite{HughesPearson, LutharPassi1989, LutharTrama1991, HertweckA6, Gildea2013, KimmerleKonovalov2013, HertweckPA, BachleMargolis2014, MargolisThesis, BachleMargolis4Primary, MargolisDelRioSerrano2016}.
	Actually all the simple groups for which the Zassenhaus Conjecture has been proved are of the form $\PSL(2,p^f)$, i.e.  projective special linear groups.
	
	Suppose that $G=\PSL(2,p^f)$ with $p$ a prime integer and let $u$ be an element of order $n$ in $V(\Z G)$.
	Hertweck proved that $u$ is rationally conjugate to an element of $G$ provided that $n$ is prime different from $p$, or $n=p$ and $f\le 2$, or $n=6$ \cite{HertweckPA}. This was extended by Margolis to $p$-regular elements of prime power order \cite{Margolis2016}.
	Recently Margolis, del R\'{\i}o and Serrano have proved that $u$ is rationally conjugate to an element of $G$ if $n$ is coprime with $2p$ \cite{MargolisDelRioSerrano2016}.
	The aim of this paper is showing that the main tool used to prove these results fails for the next natural case to consider, namely when $n=2t$ with $t$ prime and greater than $4$. On the positive side the main result of the paper provides significant information on a possible counterexample to the Zassenhaus Conjecture of this kind.
	
	The tool mentioned in the previous paragraph is the HeLP Method which was introduced by Luthar and Passi in \cite{LutharPassi1989} and improved by Hertweck in \cite{HertweckPA}.
	The basic idea of the HeLP Method, for $G$ an arbitrary finite group, is as follows:
	Given an element $u\in V(\Z G)$ of order $n$, we call \emph{distribution of partial augmentations} of $u$ to the partial augmentations of the elements $u^d$ with $d$ running on the divisors of $n$.
	Let
	$$\TPA_n(G) = \{\text{Distributions of partial augmentations of elements of } G \text{ of order } n\}$$
	and
	$$\PA_n(G) = \{\text{Distributions of partial augmentations of elements of } V(\Z G) \text{ of order } n\}.$$
	By a theorem of Marciniak, Ritter, Sehgal and Weiss, every element of $V(\Z G)$ of order $n$ is rationally conjugate to an element of $G$ if and only $\TPA_n(G)=\PA_n(G)$ (see Theorem~\ref{thm:known facts about partial augmentation}.(\ref{thm:ZC1 and augmentations})).
	Calculating $\TPA_n(G)$ is very easy but, unfortunately, calculating $\PA_n(G)$ is usually difficult.
	The HeLP Method consists in calculating a set $\VPA_n(G)$ containing $\PA_n(G)$ which will be defined in Section~\ref{Preliminaries And Main Result}.
	We call the elements of $\VPA_n(G)$, \emph{distributions of virtual partial augmentations} of order $n$ for $G$, because they satisfy some properties known for distributions of partial augmentations of elements of order $n$ in $V(\Z G)$.
	If $\VPA_n(G)=\TPA_n(G)$ then all elements of $V(\Z G)$ of order $n$ are rationally conjugate to elements of $G$ and, in case this holds for all the possible orders $n$, then the Zassenhaus Conjecture holds for $G$.
	The HeLP Method fails to verify the Zassenhaus Conjecture when $\VPA_n(G)\ne \TPA_n(G)$ for some $n$.
	Nevertheless, each element of $\VPA_n(G)\setminus \TPA_n(G)$ provides relevant information of a possible counterexample to the Zassenhaus Conjecture.
	Indeed, it determines a conjugacy class in the normalized units of $\Q G$ formed by elements with integral partial augmentations. Actually, using the representation theory of $\Q G$ one can find a concrete representative of this class. 
	With this information at hand, to prove the Zassenhaus Conjecture one should prove that none of these conjugacy classes contains an element with integral coefficients and to disprove it one should find an element in this class with integral coefficients.
	In this paper we prove that for $G=\PSL(2,p^f)$ and $t$ an odd prime different from $p$ the elements of $\VPA_{2t}(G)\setminus \TPA_{2t}(G)$ are of a very specific form (see Theorem~\ref{main}). So, it could be potentially used to prove or disprove the Zassenhaus Conjecture in this case, but completing this would require new techniques.

	The paper is organized as follows: In Section~\ref{Preliminaries And Main Result} we introduce the basic notation of the paper, explain the HeLP Method and state the main result of the paper. In Section~\ref{SectionVirtual} we collect some general technical results. In Section~\ref{SectionAccumulated} we recall the representation theory of a projective special linear group $G=\PSL(2,q)$ and calculate the sums $\widetilde{\pa}_d(m)=\sum_{g^G,|g|=m} \pa_d(g)$ for $\pa \in \VPA_{rt}( G)$ for $r$ and $t$ two different primes not dividing $q$ and $d\mid rt$. We call these sums the \emph{accumulated virtual partial augmentations}.
	Finally, in Section~\ref{SectionProof} we prove Theorem~\ref{main}.

	\section{The HeLP Method and the main result}\label{Preliminaries And Main Result}
	
	In this section we introduce the basic notation and state the main result of the paper.
	
	Let $\Z_{\ge 0}$ denote the set of non-negative integers.
	For a positive integer $n$ we always use $\zeta_{n}$ to denote a complex primitive $n$-th root of unity.
	If $F/K$ is an extension of number fields then $\Tr_{F/K}:F\rightarrow K$ denotes the trace map.
	
	Let $G$ be a finite group and let $g\in G$. Then $|g|$ denotes the order of g, $\GEN{g}$ denotes the cyclic group generated by $g$ and $g^{G}$ denotes the conjugacy class of $g$
	in $G$.
	If $\alpha$ is an element of a group ring of $G$ and $\alpha_{g}$ denotes the coefficient of an element $g$ of $G$, then the \emph{partial augmentation} of
	$\alpha$ at
	$g$ is
	$$\varepsilon_{g}(\alpha)=\sum_{h\in g^{G}}\alpha_{h}.$$
	
	The following theorems collect some well known facts, the first one on partial augmentations of torsion elements in $V(\Z G)$ and the second one on torsion elements of $V(\Z \PSL(2,q))$. See \cite[1.5, 7.3, 41.5]{SehgalBookUnits} or \cite{Berman1955,HigmanPaper,MarciniakRitterSehgalWeiss} and \cite{HertweckPA,Margolis2016}.
	
	\begin{theorem}\label{thm:known facts about partial augmentation}
		Let $G$ be a finite group and let $u$ be an element of order $n$ in $V(\mathbb{Z}G)$.
		Then the following statements hold:
		\begin{enumerate}
			\item If $u\neq1$ then $\varepsilon_{1}(u)=0$ (Berman-Higman Theorem).
			\item  $n$ divides the exponent of $G$.
			\item \label{HertweckPA}
			If $g\in G$ and $\varepsilon_{g}(u)\not=0$ then the order of $g$ divides $n$.
			\item \label{thm:ZC1 and augmentations}
			The following statements are equivalent.
			\begin{enumerate}
				\item $u$ is rationally conjugate to an element of $G$.
				\item For every $d\mid n$ there exists an element $g_{0}\in G$ such that $\varepsilon_{g}(u^{d})=0$ for every $g\in G\setminus g_{0}^{G}$.
				\item $\varepsilon_{g}(u^d)\ge 0$ for every $g\in G$ and every $d\mid n$.
			\end{enumerate}
			
		\end{enumerate}
	\end{theorem}
	
	\begin{theorem}\label{thm:ConocidoPSL(2,p)}
		Let $G=\PSL(2,q)$ for an odd prime power $q$, and let $u$ be an element in $V(\Z G)$ of order $n$. Then the following statements hold:
		\begin{enumerate}
			\item\label{Spectrum}
			If either $\gcd(n,q)=1$ or $q$ is prime then $G$ has an element of order $n$.
			\item If $n$ is a prime power and $\gcd(n,q)=1$ then $u$ is rationally conjugate to an element of $G$.
		\end{enumerate}
	\end{theorem}

	If $u$ is a unit of $\C G$ of order $n$ then the \emph{distribution of partial augmentations} of $u$ is the
	list $(\pa_d)_{d\mid n}$ of class functions of $G$ given by $\pa_d(g)=\varepsilon_g(u^d)$, for every $g\in G$ and every divisor $d$ of $n$.
	It is easy to see that every conjugate to $u$ in $\C G$ has the same distribution of partial
	augmentations as $u$.
	In case $u$ and $u'$ are elements of $V(\Z G)$ of order $n$ with the same
	distribution of partial augmentations then $u$ and $u'$ are rationally
	conjugate. To explain this we need to introduce some notation.
	Let $\chi$ be an ordinary character of $G$.
	The linear span of $\chi$ to $\C G$ takes the form
	\begin{equation}\label{LinearSpan} \chi(u)=\sum_{x^{G}}\varepsilon_{x}(u)\chi(x), \quad (u\in \C G)
	\end{equation}
	where $\sum_{x^{G}}$ is an abbreviation of $\sum_{x\in T}$ for $T$ a set of representatives of the conjugacy classes of $G$.
	Let $\rho$ be a representation of $G$ affording $\chi$.
	If $z\in\mathbb{C}$ then the multiplicity of $z$ as an eigenvalue of $\rho(u)$ only depends on the character $\chi$ and it is denoted by $\mu(z,u,\chi)$.
	As $u^n=1$, every eigenvalue of $\rho(u)$ is of the form $\zeta_n^l$ for some integer $0\le l \le n-1$, and the following formula gives the multiplicity of $\zeta_n^l$ as an eigenvalue of $\rho(u)$ in terms of the partial augmentations (see \cite{LutharPassi1989}):
	\begin{equation}\label{MultiplicitiesFormula}
	\mu(\zeta_n^{l},u,\chi)=\frac{1}{n}\sum_{x^{G}}\sum_{d|n}\varepsilon_{x}(u^{d})\Tr_{\Q (\zeta_n^{d})/\Q }(\chi(x)\zeta_n^{-ld}).
	\end{equation}
	Observe that the right side of the previous formula makes sense because if $\varepsilon_{x}(u^{d})\ne 0$ then $x^{\frac{n}{d}}=1$, by Theorem~\ref{thm:known facts about partial augmentation}.(\ref{HertweckPA}), and hence $\chi(x)\in \Q (\zeta_n^{d})$.
	Since $u$ and $u'$ have the same distribution of partial augmentations, Formula (\ref{MultiplicitiesFormula}) implies that the images of $u$ and $u'$ by all the representations of $G$ have the same eigenvalues with the same multiplicities. Thus $u$ and $u'$ are conjugate in $\C G$ and hence they are rationally conjugate (see \cite[Lemma~37.5]{SehgalBookUnits}), as desired.
	This serves as a proof of Theorem~\ref{thm:known facts about partial augmentation}.\eqref{thm:ZC1 and augmentations} and also explains why proving the Zassenhaus Conjecture for $G$ is equivalent to prove that $\PA_n(G)=\TPA_n(G)$ for every $n$.
	Clearly,
	$$\TPA_n(G)= \left\{ (\pa_d)_{d\mid n} : \text{there exists } g\in G \text{ with } \pa_d(h) =1
	\text{ if } h\in (g^d)^G \text{ and } \pa_d(h)=0 \text{ otherwise} \right\},$$
	but $\PA_n(G)$ is usually hard to calculate.

	The equality in (\ref{MultiplicitiesFormula}) extends to Brauer characters modulo a prime $p$ not dividing $n$.
	More precisely, let $\rho$ be a representation of $G$ in characteristic $p$ and let $\chi$ be the Brauer character afforded by $\rho$ with respect to a sufficiently large $p$-modular system. 
	Let $u$ be an element of order $n$ in $V(\Z G)$ with $n$ coprime with $p$ and let $\mu(\zeta_{n}^{l},u,\chi)$ denote the multiplicity of  $\overline{\zeta_{n}}^{l}$ as an eigenvalue of
	$\rho(\overline{u})$, where the bar notation stands for reduction modulo $p$.
	This multiplicity can be calculated using (\ref{MultiplicitiesFormula}) (see \cite{HertweckPA}).
	Again the formula makes sense because if $\varepsilon_g(u)\ne 0$ then $g$ is $p$-regular, by Theorem~\ref{thm:known facts about partial augmentation}.(\ref{HertweckPA}).
	
	Formula (\ref{MultiplicitiesFormula}) is the bulk of the HeLP Method.
	Namely, if $u\in V(\Z G)$ satisfies $u^n=1$ and $\chi$ is either an ordinary character of $G$ or a Brauer character of $G$ module a prime not dividing $n$,
	then for every integer $l$
	\begin{equation*}
	\frac{1}{n}\sum_{x^{G}}\sum_{d|n}\varepsilon_{x}(u^{d})\Tr_{\Q (\zeta_n^{d})/\Q }(\chi(x)\zeta_n^{-ld}) \;
	\in \; \Z_{\ge 0}.
	\end{equation*}
	
	Let $n$ be a positive integer.
	A \emph{distribution of virtual partial augmentations} of order $n$ for $G$ is a list $\pa=(\pa_d)_{d\mid n}$,
	indexed by the divisors of $n$, where each $\pa_d$ is a class function of $G$ taking values on $\Z$, and the following conditions hold:
	\begin{itemize}
		\item[(V1)] $\sum_{x^{G}} \pa_d(x)=1$;
		\item[(V2)] if $d\ne n$ then $\pa_d(1)=0$;
		\item[(V3)] if $\frac{n}{d}$ is not multiple of $|x|$ then $\pa_d(x)=0$;
		\item[(V4)] if $\chi$ is either an ordinary character of $G$ or a Brauer character of $G$ modulo a prime not dividing
		$n$ and $l\in \Z$  then
		$$\mu(\zeta_n^l,\pa,\chi)=\frac{1}{n}\sum_{x^{G}}\sum_{d|n} \pa_d(x)\Tr_{\Q (\zeta_n^{d})/\Q }(\chi(x)\zeta_n^{-ld})\; \in\; \Z_{\ge 0}.$$
	\end{itemize}
	As for (\ref{MultiplicitiesFormula}), the right side of the previous formula makes sense because, by (V3), if $\pa_d(x)\ne 0$ then the order of $x$ divides $\frac{n}{d}$ and hence $x$ is $p$-regular and $\chi(x) \in \Q(\zeta_n^d)$.
	Let 
	$$\VPA_n(G) = \{\text{Distributions of virtual partial augmentations of order } n \text{ for } G\}.$$
	We have 
	$$\TPA_n(G) \subseteq \PA_n(G) \subseteq \VPA_n(G).$$
	Indeed, the first inclusion is obvious and the second one follows from the Bergman-Higman Theorem, Theorem~\ref{thm:known
		facts about partial augmentation}.(\ref{HertweckPA}) and formula (\ref{MultiplicitiesFormula}).
	
	Let $G=\PSL(2,q)$ with $q$ an odd prime power and let $t$ be an odd prime.
	By Theorem~\ref{thm:ConocidoPSL(2,p)}.(\ref{Spectrum}), $V(\Z G)$ has elements of order $2t$ if and only if so does $G$ if and only if
	$q\equiv \pm 1 \bmod 4t$. Thus we assume that $q\equiv \pm 1 \bmod 4t$.
	For $g_0\in G$ with $|g_0|=2t$ and $t\ge 5$ let  $\pa^{(g_0)}$ denote the list of class functions $(\pa^{(g_0)}_1,\pa^{(g_0)}_2,\pa^{(g_0)}_t,\pa^{(g_0)}_{2t})$ of $G$
	defined as follows:
	\begin{equation}\label{VPASpecial}
	\pa^{(g_0)}_d(g) = \begin{cases}
	1, & \text{if } (d,g^G) \in \left\{(2t,1^G), (t,(g_0^t)^G), (2,(g_0^2)^G),
	(1,(g_0^{\frac{t-1}{2}})^G), (1,(g_0^{\frac{t+1}{2}})^G)\right\}; \\
	-1, & \text{if } (d,g^G)=(1,(g_0^{t-1})^G); \\
	0, & \text{otherwise}.
	\end{cases}
	\end{equation}
	As explained in the introduction we calculate $\VPA_{2t}(G)$, namely we prove the following.
	
	\begin{theorem}\label{main}
		Let $t$ be an odd prime and let $q$ be a prime power such that $q \equiv \pm 1 \mod 4t$ and let $G=\PSL(2,q)$.
		Then $\VPA_{6}(G)=\TPA_6(G)$ and if $t\ge 5$ then
		$$\VPA_{2t}(G)=\TPA_{2t}(G)\cup \{\pa^{(g_0)} : g_0\in G, |g_0|=2t\}.$$
	\end{theorem}
	
	Suppose that $t\ge 5$ and let $g_0$ be an element of order $2t$ in $G=\PSL(2,q)$.
	Then $\pa^{(g_0)}$ is the distribution of partial augmentations of the elements of a conjugacy class $C$ in
	the units of $\Q G$ of an element of order $2t$ in $V(\Q G)$ with integral partial augmentations.
	To settle the Zassenhaus Conjecture in this case it remains to decide whether $C$ contains an element $u$ with integral coefficients. If not, the Zassenhaus Conjecture holds in this case and otherwise $u$ provides a counterexample for the Zassenhaus Conjecture. The smallest example of this situation is encountered for $q=19$ and $t=5$. However, B\"{a}chle and Margolis has proved the Zassenhaus Conjecture for this example with a technique which they called the Lattice Method \cite{BachleMargolis2014}. Unfortunately, the Lattice Method does not apply for the next cases ($q=27$ and $q=29$ and $t=7$) basically because the representation type appearing in these cases is wild.

	\section{Preliminary general results}\label{SectionVirtual}
	
	In this section we collect some technical results that will be used in subsequent sections.
	We start quoting:
	
	\begin{lemma}\cite[Lemma~2.1]{Margolis2016}
		\label{lema Margolis}
		If $n$ and $d$ are positive integers with $d\mid n$ then
		$\Tr_{\Q (\zeta_{n})/\Q }(\zeta_{d})=\mu(d)\frac{\varphi(n)}{\varphi(d)}$,
		where $\varphi$ denotes the Euler's totient function and $\mu$ denotes the M\"{o}bius function.
	\end{lemma}

	The following well known formula, for $k$ and $d$ integers with $k>0$, will be used in several situations:
	\begin{equation}\label{SumaPotencias}
	\sum_{i=0}^{k-1} \zeta_k^{-id} = \begin{cases} 0, & \text{if } k\nmid d; \\ k, & \text{otherwise.} \end{cases}
	\end{equation}
	
	In the remainder of the section $G$ is a finite group, $n$ is a positive integer and $\pa\in\VPA_n(G)$.
	If $m$ divides $n$ then we define
	$$\pa^{\frac{n}{m}}=((\pa^{\frac{n}{m}})_d)_{d\mid m} \quad \text{with} \quad (\pa^{\frac{n}{m}})_d(g)=\pa_{d\frac{n}{m}}(g)\quad
	\text{for}\quad g\in G.$$
	Observe that if $\pa$ is the distribution of partial augmentations of an element $u\in V(\Z G)$ of order $n$ then $\pa^{\frac{n}{m}}$ is the distribution of partial augmentations of $u^{\frac{n}{m}}$.
	Moreover,
	\begin{equation}\label{pa/m}
	\text{if }\pa\in \VPA_n(G) \text{ and } m\mid n \text{ then } \pa^{\frac{n}{m}}\in \VPA_m(G).
	\end{equation}
	Indeed, that $\pa^{\frac{n}{m}}$ satisfies (V1), (V2) and (V3) is elementary and (V4) follows from the following lemma.
	
	\begin{lemma}\label{PADivided}
		Let $G$ be a finite group. Let $n$ and $m$ be positive integers with $m\mid n$, let $l\in \Z$ and let $\pa \in \VPA_n(G)$.  Let $\chi$ be either an ordinary character of $G$ or a Brauer character of $G$ module a prime not dividing $n$.  Then
		$$\mu(\zeta_m^l,\pa^{\frac{n}{m}},\chi)=\sum_{\xi,\xi^{\frac{n}{m}}=\zeta_m^l}\mu(\xi,\pa,\chi).$$
	\end{lemma}
	
	\begin{proof}
		Let $k=\frac{n}{m}$ and fix $\xi_0 \in \C$ with $\xi_0^k=\zeta_m^l$. Then $\xi^k=\zeta_m^l$ if and only if $(\xi\xi_0^{-1})^k=1$ if and only if $\xi=\xi_0 \zeta_k^i$ for some $i\in \{0,1,\dots,k-1\}$.
		Then
		\begin{eqnarray*}
			\sum_{\xi,\xi^k=\zeta_m^l} \mu(\xi,\pa,\chi) &=&
			\frac{1}{n} \sum_{x^G} \sum_{d\mid n} \pa_d(x)\Tr_{\Q(\zeta_n^d)/\Q}\left(\chi(x)\xi_0^{-d}\sum_{i=0}^{k-1} \zeta_k^{-id}\right) \\
			&=&
			\frac{1}{m} \sum_{x^G} \sum_{d\mid n,k\mid d} \pa_d(x)\Tr_{\Q(\zeta_n^d)/\Q}(\chi(x)\xi_0^{-d}),
		\end{eqnarray*}
		where the last equality is a consequence of (\ref{SumaPotencias}).
		Furthermore $\{d:d\mid n, k\mid d\} = \{kd_{1} : d_{1}\mid m\}$ and $\zeta_n^k$ has order $m$.
		Thus
		$$
		\sum_{\xi,\xi^k=\zeta_m^l} \mu(\xi,\pa,\chi) =
		\frac{1}{m} \sum_{x^G} \sum_{d_{1}\mid m} (\pa^{\frac{n}{m}})_{d_1}(x)\Tr_{\Q(\zeta_{m}^{d_1})/\Q}(\chi(x)\zeta_{m}^{-ld_1})
		= \mu(\zeta_m^l,\pa^{\frac{n}{m}},\chi).
		$$
	\end{proof}
	
	\begin{proposition}\label{PANonEmptyDivides}
		Let $G$ be a finite group and let $n$ be a positive integer.
		If $\VPA_n(G)\ne \emptyset$ then every prime divisor of $n$ divides $|G|$.
	\end{proposition}
	\begin{proof}
		Let $p$ be a prime not dividing $|G|$ and let $\pa\in \VPA_p(G)$.
		By (V2) and (V3), $\pa_1(g)=0$ for every $g\in G$ and this is in contradiction with (V1).
		This shows that if $p$ does not divides the order of $G$ then $\VPA_p(G)=\emptyset$.
		Now, if $\pa\in \VPA_n(G)$, with $p$ a prime divisor of $n$ then $\pa^{\frac{n}{p}}\in \VPA_p(G)$ by \eqref{pa/m}, and hence $p$ divides the order of $G$, by the previous sentence.
	\end{proof}

	Let $\chi$ be either an ordinary character of $G$ or a Brauer character of $G$ module a prime $p$ not dividing $n$ and let $m$ be a divisor of $n$.
	Let
	$$\mu_m^-(\pa,\chi) = \frac{1}{n}\sum_{x^{G}}\sum_{d\mid n,m\mid d}\pa_d(x)\Tr_{\Q (\zeta_n^{d})/\Q }(\chi(x)).$$
	By (V4), for every $l\in \Z$, we have
	$$
	0 \le \mu(\zeta_m^l,\pa,\chi) =
	\mu_m^-(\pa,\chi) + \frac{1}{n}\sum_{x^{G}}\sum_{d\mid n,m\nmid d}\pa_d(x)\Tr_{\Q (\zeta_n^{d})/\Q }(\chi(x)\zeta_m^{-dl}).
	$$
	Combining this with (\ref{SumaPotencias}) we obtain
	\begin{eqnarray*}
		0 & \le & \mu(1,\pa,\chi) = \mu_m^-(\pa,\chi) + \frac{1}{n}\sum_{x^{G}}\sum_{d\mid n,m\nmid d}\pa_d(x)\Tr_{\Q (\zeta_n^{d})/\Q }(\chi(x)) \\
		&=&
		\mu_m^-(\pa,\chi) - \frac{1}{n}\sum_{l=1}^{m-1}\sum_{x^{G}}\sum_{d\mid n,m\nmid d}\pa_d(x)\Tr_{\Q (\zeta_n^{d})/\Q }(\chi(x)\zeta_m^{-dl}) \le m \mu_m^-(\pa,\chi).
	\end{eqnarray*}
	We record this for future use:
	\begin{equation}\label{mu0}
	0\leq\mu(1,\pa,\chi) \leq m\mu_{m}^{-}(\pa,\chi).
	\end{equation}
	
	\section{Preliminary results for PSL$(2,q)$}\label{SectionAccumulated}

	In this section $p$ is an odd prime, $q=p^f$ with $f\ge 1$ and $G=\PSL(2,q)$.
	In the first part of the section we describe the ordinary and Brauer characters of $G$ which will be used in the remainder of the paper.
	In the second part we first describe $\VPA_{r}(G)$ with $r$ a prime different from $p$ and then we calculate the accumulated virtual partial augmentations of elements in $\VPA_{rt}(G)$ with $r$ and $t$ different primes such that $q\equiv \pm 1 \bmod 2rt$.
	
	Suppose that $G$ has an element $g_0$ of order $m$ with $p\nmid m$.
	Then $q\equiv \epsilon \bmod 2m$ with $\epsilon=\pm 1$.
	Moreover, every element of $G$ of order dividing $m$ is conjugate to some element of $\GEN{g_0}$ and $g_0^i$ and $g_0^j$ are conjugate in $G$ if and only if $i\equiv \pm j \bmod m$.
	
	If $h$ is an integer then let $\phi_h,\psi_h:\GEN{g_0}\rightarrow \C$ be defined as follows:
	$$\phi_h(g_0^i)=\begin{cases}
	q+\epsilon, & \text{if } m \mid i; \\ \epsilon(\zeta_{m}^{hi}+\zeta_{m}^{-hi}), & \text{otherwise;}
	\end{cases} \quad \quad
	\psi_h(g_0^i)=\begin{cases}
	q-\epsilon, & \text{if } m \mid i; \\ 0, & \text{otherwise.}
	\end{cases}$$
	Given $R=(r_0,r_1,\dots,r_k)\in \Z^{k+1}$, let
	\begin{equation*}
	X_{R}=\{(s_0,s_1,\dots,s_k)\in \Z^{k+1} : -r_j \le s_j \le r_j \text{ and } 2\mid r_j - s_j \text{ for each } j\}.
	\end{equation*}
	If moreover $r_0+\dots+r_k$ is even then let
	\begin{equation}\label{Vrh}
	V_{R;h}=\left\{(s_0,s_1,\dots,s_k)\in X_{R}\setminus \{(0,\dots,0)\} : \frac{\sum_{j=0}^k s_jp^j}{2} \equiv \pm h \bmod m\right\}
	\end{equation}
	and let $\chi_{r_0,r_1,\dots,r_k}:\GEN{g_0}\rightarrow \C$ be defined by
	$$\chi_{r_0,r_1,\dots,r_k}(g_0^i) =
	\sum_{(s_0,s_1,\dots,s_k)\in X_{R}} \zeta_{m}^{i \frac{\sum_{j=0}^k s_jp^j}{2}}.$$
	
	\begin{lemma}\label{LemmaCharacters}
		Let $g_0$ be a $p$-regular element of $G$ of order $m$.
		Then the following statements hold for $s\in \Z$ and $R=(r_0,r_1,\dots,r_k)\in \Z^{k+1}$ with $r_0+r_1+\dots+r_k$ even:
		\begin{enumerate}
			\item\label{PhiPsi} If $m\nmid s$ then both $\phi_s$ and $\psi_s$ are the restriction to
			$\GEN{g_0}$ of ordinary characters of $G$.
			\item\label{Chis} $\chi_{r_0,r_1,\dots,r_k}$ is the restriction to $\GEN{g_0}$ of a Brauer character of $G$ modulo $p$.
			\item\label{ChisIrreducible} If $\chi$ is an irreducible Brauer character of $G$ modulo $p$ then the restriction of $\chi$ to $\GEN{g_0}$ equals
			$\chi_{r_0,r_1,\dots,r_{f-1}}$ for some integers $0\le r_0,r_1,\dots,r_{f-1} \le p-1$ with
			$r_0+r_1+\dots+r_{f-1}$ even.
			\item\label{ChisInPhisPsis} 
			$$\chi_{r_0,r_1,\dots,r_k}=\begin{cases}
			(1+2n_0) 1_G + \epsilon \sum_{h=1}^{\frac{m}{2}} n_h \left( \phi_h-\psi_h\right), & \text{if } 2 \mid r_j \text{ for all } j; \\
			2n_0 1_G + \epsilon \sum_{h=1}^{\frac{m}{2}} n_h \left( \phi_h-\psi_h\right), & \text{otherwise};\end{cases}$$
			where $2n_h=|V_{R;h}|$ and $1_G$ denotes the trivial character of $G$.
		\end{enumerate}
	\end{lemma}
	
	\begin{proof}
		(\ref{PhiPsi})
		To prove that $\phi_s$ and $\psi_s$ are the restriction to $\GEN{g_0}$ of ordinary characters of $G$ we simply express them in terms of the irreducible characters
		$\eta_1$, $\eta_2$, $\theta_i$ and $\chi_i$ of $G$ as described in Table~2 of \cite{HertweckPA}.
		If $s \equiv \pm s' \bmod m$ then $\phi_s=\phi_{s'}$ and $\psi_s=\psi_{s'}$. Therefore we may assume that $1\le s \le \frac{m}{2}$.
		Firstly, if $m$ is even then $\phi_{\frac{m}{2}}$ is the restriction of $\eta_1+\eta_2$, and $\psi_{\frac{m}{2}}$
		is the restriction of any $\theta_j$ if $\epsilon=1$, and the restriction of any $\chi_i$ if
		$\epsilon=-1$.
		This covers the case $s=\frac{m}{2}$. Suppose otherwise that $1\le s<\frac{m}{2}$.
		If $\epsilon=1$ then $\phi_s$ is the restriction of $\chi_{s\frac{q-1}{2m}}$ and $\psi_s$ is the restriction of $\theta_{s\frac{q-1}{2m}}$, while if
		$\epsilon=-1$ then $\phi_s$ is the restriction of $\theta_{s\frac{q+1}{2m}}$ and $\psi_s$ is the restriction of
		$\chi_{s\frac{q+1}{2m}}$.
		
		(\ref{Chis}) Let $K$ be a field of characteristic $p$. The following defines an action by $K$-automorphisms on the group $\SL(2,q)$ on the ring of polynomials $K[X,Y]$
		\cite[Pages~14--16]{Alperin1986}:
		$$\pmatriz{{cc} a & b \\ c & d} X = aX+bY, \quad \pmatriz{{cc} a & b \\ c & d} Y = cX+dY.$$
		If $n$ is a positive integer then the vector space $V_n$ formed by the homogenous polynomials of degree $n$ is invariant under this action
		and, if moreover $n$ is even then $\pmatriz{{cc} -1 & 0 \\ 0 & -1}$ acts trivially on $V_n$. Therefore $G$ acts on $V_n$
		provided that $n$ is even.

		Let us fix integers $0\le r_0,r_1,\dots,r_k \le p-1$ such that $r_0+r_1+\dots+r_k$ is even and let $n=r_0+r_1p+\dots+r_kp^k$.
		Then $n$ is even.
		Let $W_{r_0,r_1,\dots,r_k}$ be the subspace of $V_n$ generated by the polynomials of the form $X^iY^{n-i}$ with $i=i_0+i_1p+\dots+i_kp^k$ and $0\le i_j \le
		r_j$ for every $j$. For such $s$ we have
		\begin{eqnarray*}
			\pmatriz{{cc} a & b \\ c & d} X^iY^{n-i} &=& (aX+bY)^i(cX+dY)^{n-i} \\ &=&
			\prod_{h=0}^k \left(a^{p^h} X^{p^h}+b^{p^h}Y^{p^h}\right)^{i_h}
			\left(c^{p^h} X^{p^h}+d^{p^h}Y^{p^h}\right)^{r_h-i_h}
			\\ &=&
			\prod_{h=0}^k \left(\sum_{u=0}^{i_h} \alpha_{h,u} X^{up^h} Y^{(i_h-u)p^h}\right)
			\left(\sum_{v=0}^{r_h-i_h} \beta_{h,v} X^{vp^h} Y^{(r_h-i_h-v)p^h}\right)
			\\ &=&
			\prod_{h=0}^k \left(\sum_{j=0}^{r_h} \gamma_{h,j} X^{jp^h} Y^{(r_h-j)p^h}\right)
			=
			\sum_{\stackrel{j=j_0+j_1p+\dots+j_kp^k}{0\le j_h \le r_h}} \delta_j X^j Y^{n-j} \in W_{r_0,\dots,r_k}.
		\end{eqnarray*}
		Therefore, $W_{r_0,r_1,\dots,r_k}$ is invariant by the action of $G$ and hence it is a $KG$-module.
		Let $\rho$ denote the $K$-representation of $G$ given by $W_{r_0,r_1,\dots,r_k}$ and let
		$\chi$ be the Brauer character associated to the $p$-modular character afforded by $\rho$.
		If $\epsilon=1$ then $\overline{\zeta_{2m}}$ belongs to the field with $q$ elements, so that the diagonal matrix
		$D=\diag(\overline{\zeta_{2m}},\overline{\zeta_{2m}}^{-1})$ belongs to $\SL(2,q)$.
		After a suitable election of $\zeta_{2m}$ we may assume that $g_0=D$ because $g_0$ is conjugate to a power of $D$ in $G$.
		Then each base element $X^sY^{n-s}$ is an eigenvector of $\rho(g_0)$ with eigenvalue
		$\overline{\zeta_{2m}}^{2i-n}=\overline{\zeta_{2m}}^{s}=\overline{\zeta_m}^{\frac{s}{2}}$ for $s=s_0+s_1p+\dots+s_kp^k$, $-r_j \le s_j \le r_j$ and
		$2\mid r_j-s_j$ for each $j$. Therefore $\chi_{r_0,r_1,\dots,r_k}$ coincides with the restriction of $\chi$ to $\GEN{g_0}$.
		Suppose that $\epsilon=-1$. Then $D\in \SL(2,q^2)$ and this group acts on $W_{r_0,r_1,\dots,r_k}$ in the same way.
		Again we may assume that $g_0=D$ and the same argument shows that the restriction of $\chi$ to $\GEN{g_0}$ coincides with $\chi_{r_0,r_1,\dots,r_k}$.

		(\ref{ChisIrreducible}) The absolutely irreducible characters in characteristic $p$ of $\SL(2,q)$ have been described in \cite{BrauerNesbitt} (see also
		\cite{Srinivasan}).
		After lifting these characters to $G$ we obtain that if $0\le r_0,r_1,\dots,r_{f-1}\le p-1$ and
		$r_0+r_1+\dots+r_{f-1}$ is even then $\chi_{r_0,r_1,\dots,r_{f-1}}$ is the restriction to $\GEN{g_0}$ of an irreducible Brauer character of $G$ modulo $p$
		and, conversely, the restriction to $\GEN{g_0}$ of any irreducible Brauer character of $G$ modulo $p$ is of this form.
		
		(\ref{ChisInPhisPsis}) Straightforward.
	\end{proof}
	
	\textbf{Convention}: In the remainder of the paper we will often encounter some fixed element $g_0\in G$. Then we will use the functions $\phi_h,\psi_h,\chi_{r_0,r_1,\dots,r_k} : \GEN{g_0}\rightarrow \C$ and we will abuse the notation by referring to the first two as ordinary characters of $G$ and to the last ones as the Brauer characters of $G$ module $p$, rather as the restriction to $\GEN{g_0}$ of such an ordinary or Brauer character. This will be harmless because we  will use only these restrictions.
	
	\begin{proposition}\label{TPAr=PAr}
		Let $G=\PSL(2,q)$ with $q$ an odd prime power and let $r$ be a prime not dividing $q$.
		Then $\TPA_r(G)=\VPA_r(G)$.
	\end{proposition}
	
	\begin{proof}
		The result is trivial if $q \not \equiv \pm 1 \bmod 2r$ because in such case $r$ does not divides $|G|$ and hence $\VPA_{r}(G)=\emptyset=\TPA_r(G)$, by Proposition~\ref{PANonEmptyDivides}.
		The result is also trivial if $r=2$ because all the elements of $G$ of order $2$ are conjugate.
		So suppose that $r$ is odd and $q\equiv \pm 1 \bmod 2r$.
		
		Let $\pa\in \VPA_r(G)$.
		To prove the lemma we fix an element $g_0\in G$ of order $r$ and use (V4) with the Brauer character $\chi_{2}$.
		We have that $\left\{g_0^i:i=1,\dots,\frac{r-1}{2}\right\}$ is a set of representatives of the conjugacy classes of $G$ of elements of order $r$.
		Then, by (V1), (V2) and (V3), we have $\sum_{j=1}^{\frac{r-1}{2}}{\pa_1(g_0^{j})}=1$.
		Moreover, for each  $l, i\in \{1,\dots,\frac{r-1}{2}\}$ we have
		$$\Tr_{\Q (\zeta_r)/\Q }(\chi_{2}(g_0^{i})\zeta_{r}^{-l})=\begin{cases}
		r-3, & \text{if } i\equiv \pm l \bmod r; \\ -3, & \text{otherwise}.
		\end{cases}$$
		Hence
		$$
		\mu(\zeta_{r}^{l},\pa,\chi_{2})
		=\frac{1}{r}\left((r-3)\pa_1(g_0^{l})-3\sum_{j=1, j\ne l}^{\frac{r-1}{2}}{\pa_1(g_0^{j})+3}\right)=\pa_1(g_0^{l})\; \in \; \Z_{\ge 0}.
		$$
		Therefore, there is an integer $i$ in the interval $\left[1,\frac{r-1}{2}\right]$ such that $\pa_1(g_0^i)=1$ and $\pa_1(h)=0$ for every $h\in G\setminus (g_0^i)^G$.
		Then $\pa$ is the distribution of partial augmentations of $g_0^i$.
		We conclude that $\pa\in\TPA_r(G)$.
	\end{proof}

	\begin{corollary}\label{paPrimo}
		Let $G=\PSL(2,q)$ with $q$ an odd prime power and let $m$ be a square-free positive integer.
		Assume that $q \equiv \pm 1 \bmod 2m$ and let $\pa\in \VPA_{m}(G)$.
		Then $G$ has an element $g_0$ of order $m$ such that for every prime divisor $t$ of $m$ we have
		\begin{equation}\label{patGeneral}
		\pa_{\frac{m}{t}}(g) = \begin{cases} 1, & \text{if } g \in \left(g_0^{\frac{m}{t}}\right)^G; \\ 0, &
		\text{otherwise}.\end{cases}
		\end{equation}
	\end{corollary}
	
	\begin{proof}
		Fix an element $g_1$ of $G$ of order $m$. By \eqref{pa/m} and Proposition \ref{TPAr=PAr}, if $t$ is a prime divisor of $m$ then $\pa^{\frac{m}{t}}\in \VPA_t(G)=\TPA_t(G)$. As every element of order $t$ in $G$ is conjugate in $G$ to an element of $\GEN{g_1^{\frac{m}{t}}}$, we deduce that there is an integer $i_t$ coprime with $t$ such that $\pa^{\frac{m}{t}}$ is the distribution of partial augmentations of $g_1^{\frac{m}{t}i_t}$.
		Let $i$ be an integer with $i\equiv \pm i_t\bmod t$ for every prime $t$ dividing $m$ and let $g_0=g_1^i$.
		Then $g_0^{\frac{m}{t}}=g_1^{\frac{m}{t}i_t}$ for every prime $t$ and hence $\pa^{\frac{m}{t}}$ is the distribution of partial augmentations of $g_0^{\frac{m}{t}}$.
		In particular, \eqref{patGeneral} holds for every prime $t$.
	\end{proof}
	
	In the remainder of the section $r$ and $t$ are different primes such that $q \equiv \pm 1 \bmod 2rt$ and $\pa\in \VPA_{rt}(G)$.
	By Corollary \ref{paPrimo}, $G$ has an element $g_{0}$ of order $rt$, which will be fixed for the remainder of the section, such that
	\begin{equation}\label{par}
	\pa_r(g) = \begin{cases} 1, & \text{if } g \in  (g_0^{r})^G; \\ 0, &
	\text{otherwise};\end{cases}
	\quad \text{and} \quad 
	\pa_t(g) = \begin{cases} 1, & \text{if } g \in (g_0^{t})^G; \\ 0, &
	\text{otherwise}.\end{cases}
	\end{equation}
	As every element of $G$ of order divisible by $rt$ is conjugate to an element of $\langle g_{0}\rangle$, we have that
	$\pa_1(g)=0$ for every $g\in G$ which is not conjugate to an element of $\langle g_{0}\rangle\setminus\{1\}$.
	This will simplify the expression in (V4), as in the following lemma.
	
	For a general finite group $G$, an element $\pa =(\pa_d)_{d\mid n}$ of $\VPA_n(G)$ and a positive integer $m$ we define the \emph{accumulated virtual partial augmentations} of $\pa$ at $m$ as follows:
	$$\widetilde{\pa}_{d}(m) = \sum_{g^G,|g|=m} \pa_d(g).$$
	
	\begin{lemma}\label{mu0 2}
		Let $\chi$ be either an ordinary character or a Brauer character of $G$ module $p$.
		Then
		\begin{eqnarray*}
			\mu(1,\pa,\chi)&=&
			\frac{1}{rt} \left[\widetilde{\pa}_1(rt)\Tr_{\Q(\zeta_{rt})/\Q}(\chi(g_{0}))+
			\widetilde{\pa}_1(t) \Tr_{\Q(\zeta_{rt})/\Q}(\chi(g_{0}^{r}))+
			\widetilde{\pa}_1(r) \Tr_{\Q(\zeta_{rt})/\Q}(\chi(g_{0}^{t})) + \right.\\
			& & \left. \hspace{1cm} \Tr_{\Q (\zeta_{r})/\Q }(\chi(g_{0}^{t}))+\Tr_{\Q (\zeta_{t})/\Q }(\chi(g_{0}^{r}))+\chi(1) \right]
		\end{eqnarray*}
		and
		$$0\leq \mu(1,\pa,\chi) \leq \frac{1}{t} \left[ \chi(1)+\Tr_{\Q (\zeta_{t})/\Q }(\chi(g_{0}^{r}))\right]
		.$$
	\end{lemma}
	
	\begin{proof}
		If $x$ and $y$ are elements of $G$ with the same order, then $x$ is conjugate in $G$ to a power of $y$.
		This implies that if $e$ is a multiple of this common order then there exists $\sigma\in\text{Gal}(\Q (\zeta_{e})/\Q )$ such that
		$\chi(x)=\sigma(\chi(y))$. (If $\chi$ is a Brauer character modulo $p$ then one assume that $x,y\in G_p$.)
		Thus $\Tr_{\Q (\zeta_{e})/\Q }(\chi(x))=\Tr_{\Q (\zeta_{e})/\Q }(\chi(y))$.
		Combining this with (V4) we obtain the expression for $\mu(1,\pa,\chi)$ in the lemma.
		Moreover,
		$\mu_r^-(\pa,\chi)=\frac{1}{rt}(\chi(1)+\Tr_{\Q (\zeta_{t})/\Q }(\chi(g_{0}^{r})))$, by \eqref{par}.
		The inequality is then a consequence of (\ref{mu0}) for $n=rt$ and $m=r$.
	\end{proof}
	
	The specialization of the following lemma to the case where $\pa$ is the distribution of partial augmentations of a torsion element of $V(\Z G)$ is a particular case of a result of Wagner
	\cite{WagnerThesis,HeLPArticulo}.
	
	\begin{lemma}\label{t divides the accumulated}
		$t$ divides $\widetilde{\pa}_1(t)$ and $r$ divides $\widetilde{\pa}_1(r)$.
	\end{lemma}
	
	\begin{proof}
		By symmetry we only have to prove $t\mid \widetilde{\pa}_1(t)$.
		We will give one proof for the case when $t$ is odd and another one for the case when $r$ is odd.
		This cover all the cases because $r$ and $t$ are different primes.
		Write $q=\epsilon +2rth$ with $\epsilon=\pm 1$ and $h$ an integer.
		We will use Lemma \ref{mu0 2} and (V4) with the ordinary character $\phi_{t}$ (relative to the element $g_0$ of order $m=rt$ fixed above).
		
		Using Lemma~\ref{lema Margolis}, for any $j$ we have:
		\[
		\Tr_{\Q (\zeta_{rt})/\Q }(\phi_t(g_{0}^{j}))=\begin{cases}
		2\epsilon(r-1)(t-1), & \text{if }r|j;\\
		2\epsilon(1-t), & \text{if }r\nmid j.
		\end{cases}
		\]
		Clearly, we also have $\Tr_{\Q (\zeta_{t})/\Q }(\phi_t(g_{0}^{j}))=2\epsilon(t-1)$ if $\gcd(rt,j)=r$, and
		$\Tr_{\Q (\zeta_{r})/\Q }(\phi_t(g_{0}^{j}))=-2\epsilon$ if $\gcd(rt,j)=t$.
		By (V1) we have  $\widetilde{\pa}_1(r)+\widetilde{\pa}_1(t)+\widetilde{\pa}_1(rt)=1$.
		Combining this  with Lemma \ref{mu0 2} we get
		$$rt \mu(1,\pa,\phi_t) =2r\epsilon(t-1)\widetilde{\pa}_1(t)+2rth.$$
		Thus, if $t$ is odd then $t$ divides $\widetilde{\pa}_1(t)$.
		
		Suppose $r$ odd.
		Let $1\leq j\leq\frac{r-1}{2}$ be an integer. We have for every integer $i$ that
		\begin{equation}
		\Tr_{\Q(\zeta_{rt})/\Q}(\phi_t(g_{0}^{i})\zeta_{rt}^{-jt})=\begin{cases}
		\epsilon(r-2)(t-1), & \text{if }i\equiv\pm j\bmod r;\\
		2\epsilon(1-t), & \text{if }i\not\equiv\pm j\bmod r.
		\end{cases}\label{eq:TrazaRT}
		\end{equation}
		For every $1\leq i\leq \frac{r-1}{2}$ we denote $x_{i}=\sum_{1\leq k\le \frac{rt}{2},k\equiv\pm i\bmod r}\pa_1(g_{0}^{k})$.
		Then $\widetilde{\pa}_1(t)+\sum_{i=1}^{\frac{r-1}{2}} x_i=1$.
		By (\ref{eq:TrazaRT}), we obtain that
		$\sum_{1\le k \le \frac{rt}{2}} \pa_1(g_0^k) \Tr_{\Q(\zeta_{rt})/\Q}(\phi_t(g_{0}^k)\zeta_{rt}^{-jt})=
		\epsilon(t-1)(-2\widetilde{\pa}_1(t)+(r-2)x_{j}-2\sum_{i\ne j}x_{i}
		)=\epsilon(t-1)(rx_{j}-2)$.
		Moreover, we have
		$\Tr_{\Q(\zeta_{t})/\Q}(\phi_t(g_0^{r})\zeta_{rt}^{-jtr})=2\epsilon(t-1)$
		and
		\[
		\Tr_{\Q(\zeta_{r})/\Q}(\phi_t(g_0^{t})\zeta_{rt}^{-jt^{2}})=
		\epsilon w_{j}\quad \text{with} \quad w_j=\begin{cases} (r-2), & \text{if } jt \equiv\pm 1 \bmod r;\\
		-2, & \text{if } jt \not\equiv\pm 1\bmod r.
		\end{cases}
		\]
		Applying condition (V4) with the character $\phi_t$ and $l=jt$ we deduce that
		\[
		\epsilon(t-1)(rx_{j}-2)+2\epsilon(t-1)+\epsilon w_{j}+2rth+2\epsilon=2hrt+\epsilon(r x_{j}(t-1)+2+w_{j})
		\]
		is a multiple of $rt$.
		In particular, $t$ divides $2+w_j-rx_j$ for every $j=1,\dots,\frac{r-1}{2}$.
		Summing for $j=1,\ldots,\frac{r-1}{2}$ and taking into account that $\sum_{j=1}^{\frac{r-1}{2}}w_{j}=1$ we obtain that $t$ divides
		$r\left(1-\sum_{j=1}^{\frac{r-1}{2}} x_j\right)=r\widetilde{\pa}_1(t)$.
		Therefore $t\mid \widetilde{\pa}_1(t)$, as desired.
	\end{proof}
	
	\begin{proposition}\label{accumulated}
		Let $\pa$ be an element of $\VPA_{rt}(G)$, where $G=\PSL(2,q)$ for an odd prime power $q$ and $r$ and $t$ are different primes with $(rt,q)=1$.
		Then 
		$$\widetilde{\pa}_1(r)=\widetilde{\pa}_1(t)=0 \quad\text{and}\quad \widetilde{\pa}_1(rt)=1.$$
	\end{proposition}
	\begin{proof}
		
		By symmetry we may assume that $r<t$.
		We first prove that $\widetilde{\pa}_1(t)=0$.
		For that we use Lemma~\ref{mu0 2} applied to the Brauer character $\chi_{2r}$.
		Using Lemma~\ref{lema Margolis} we have
		\[
		\Tr_{\Q (\zeta_{rt})/\Q }(\chi_{2r}(g_{0}))=(t-1)(r-1),
		\]
		\[
		\Tr_{\Q (\zeta_{rt})/\Q }(\chi_{2r}(g_{0}^{t}))=(t-1)(r-1),\quad\Tr_{\Q (\zeta_{rt})/\Q
		}(\chi_{2r}(g_{0}^{r}))=(r-1)(t-1-2r),
		\]
		\[
		\Tr_{\Q (\zeta_{r})/\Q }(\chi_{2r}(g_{0}^{t}))=r-1\quad\text{ and }\quad\Tr_{\Q (\zeta_{t})/\Q
		}(\chi_{2r}(g_{0}^{r}))=t-1-2r.
		\]
		Combining these equalities with
		$\widetilde{\pa}_1(r)+\widetilde{\pa}_1(t)+\widetilde{\pa}_1(rt)=1$, we obtain by straightforward calculations that
		$\mu(1,\pa,\chi_{2r})=1-2\frac{\widetilde{\pa}_1(t)}{t}$ and $\chi_{2r}(1)
		+\Tr_{\Q(\zeta_{t})/\Q }(\chi_{2r}(g_{0}^{r}))=t$.
		Using Lemma~\ref{mu0 2}, we conclude that $0\leq \frac{\widetilde{\pa}_1(t)}{t}\leq\frac{1}{2}$.
		By Lemma~\ref{t divides the accumulated}, we know that $\frac{\widetilde{\pa}_1(t)}{t}$ is an integer, hence
		$\widetilde{\pa}_1(t)=0$ as desired.
		
		Then $\widetilde{\pa}_1(r)+\widetilde{\pa}_1(rt)=1$ and it remains only to
		show that $\widetilde{\pa}_1(r)=0$.
		For that we use Lemma~\ref{mu0 2} with the Brauer character $\chi_{2}$.
		In this case we have
		\[
		\Tr_{\Q (\zeta_{rt})/\Q }(\chi_{2}(g_{0}))=tr-r-t+3,
		\]
		\[
		\Tr_{\Q (\zeta_{rt})/\Q }(\chi_{2}(g_{0}^{t}))=tr-r-3t+3,\quad\Tr_{\Q (\zeta_{rt})/\Q
		}(\chi_{2}(g_{0}^{r}))=tr-3r-t+3,
		\]
		\[
		\Tr_{\Q (\zeta_{r})/\Q }(\chi_{2}(g_{0}^{t}))=r-3\quad\text{ and }\quad\Tr_{\Q (\zeta_{t})/\Q
		}(\chi_{2}(g_{0}^{r}))=t-3.
		\]
		Therefore, we have that $\mu(1,\pa,\chi_{2})=1-2\frac{\widetilde{\pa}_1(r)}{r}$ and
		$\chi_{2}(1)+\Tr_{\Q (\zeta_{t})/\Q }(\chi_{2}(g_{0}^{r}))=t$.
		Applying Lemma~\ref{mu0 2} we obtain $0\leq\frac{\widetilde{\pa}_1(r)}{r}\leq\frac{1}{2}$ and the result follows,
		since, by Lemma~\ref{t divides the accumulated}, $\frac{\widetilde{\pa}_1(r)}{r}$ is an integer.
	\end{proof}
	
	\section{Proof of Theorem~\ref{main}}\label{SectionProof}
	
	In this section we prove Theorem~\ref{main}.
	So $t$ is an odd prime integer, $q$ is a prime power with $q\equiv \pm 1\bmod 4t$ and $G=\PSL(2,q)$.
	If $t=3$ then $G$ has a unique conjugacy class of elements of order $3$ and a unique one of elements of order $6$.
	Combining this with Proposition \ref{accumulated} we deduce that $\VPA_6(G)=\TPA_6(G)$.
	The specialization of this to partial augmentations of torsion units is a result of Hertweck \cite[Proposition~6.6]{HertweckPA}.
	So in the remainder we assume that $t\ge 5$.
	We set $n=\frac{t-1}{2}$.
	
	We first prove the inclusion $\VPA_{2t}(G)\subseteq \TPA_{2t}(G)\cup \{\pa^{(g_0)} : g_0\in G, |g_0|=2t\}$.
	For that we take an element $\pa=(\pa_1,\pa_2,\pa_t,\pa_{2t})$ of $\VPA_{2t}(G)$ and we show that there is an element $g_0$ of order $2t$ in $G$ such that either $\pa=\varepsilon^*$ or $\pa=\pa^{(g_0)}$, where $\varepsilon^*$ is the distribution of partial augmentations of $g_0$ and $\pa^{(g_0)}$ is as in \eqref{VPASpecial}.
	By (V3), if $g$ is an element of $G$ of order not dividing $2t$ then $\pa_d(g)=0$ for each $d\mid 2t$.
	By the results of Section~\ref{SectionAccumulated}, there exists $g_0\in G$ of order $2t$ such that  $\pa_t$ and  $\pa_2$ are as in (\ref{par}) with $r=2$.
	This shows that $\pa_t=\pa^{(g_0)}_t=\varepsilon^*_t$ and $\pa_2=\pa_2^{(g_0)}=\varepsilon^*_2$, where $\varepsilon^*$ is the distribution of partial augmentations of $g_0$.
	As, clearly $\pa_{2t}=\pa^{(g_0)}_{2t}=\varepsilon_{2t}^{*}$, it remains to show that $\pa_1$ is either $\varepsilon^*_1$ or $\pa^{(g_0)}_1$.
	
	Let $\Odd$ denote the set of odd integers in the interval $[1,t-1]$ and for each $l\in [1,t-1]$ let
	$$i_l=\begin{cases}
	l, & \text{if } l\in \Odd; \\ t-l, & \text{otherwise};
	\end{cases} \; w_l = \begin{cases}1, &\text{if } l\in \{1, t-1\};\\ 0, &\text{otherwise}; \end{cases}
	\; \text{and} \;
	W_l = \begin{cases}1, &\text{if } l\in \{1, 2, t-2, t-1\};\\
	0, &\text{otherwise}. \end{cases}$$
	Observe that $i_l$ is the unique element $j\in \Odd$ with $j\equiv \pm l \bmod t$ and $i_{nl}$ is the unique element $j\in \Odd$ with $2j\equiv \pm l \bmod t$. Moreover, $i_l\ne i_{nl}$ because $n\not\equiv \pm 1 \bmod t$, as $t\ge 5$.
	
	We also use the notation $T=\Tr_{\Q(\zeta_{2t})/\Q}=\Tr_{\Q(\zeta_{t})/\Q}$. 
	
	Observe that $\{g_0^i : i\in \Odd \}$ and $\{g_0^{t-i} : i\in \Odd \}$ are sets of representatives of the conjugacy classes of elements of $G$ of orders $2t$ and $t$, respectively. By Proposition~\ref{accumulated} we have
	\begin{equation}\label{ValuesAccumulated}
	\pa_{1}(g_0^t) = \sum_{i\in \Odd} \pa_1\left(g_{0}^{t-i}\right)=0 \quad \text{ and } \quad
	\sum_{i\in\Odd}\pa_1\left(g_{0}^{i}\right)=1.
	\end{equation}
	
	By (V4) and Lemma \ref{LemmaCharacters}.\eqref{Chis} if $m$ is an even integer and $l\in \Z$ then 	
	\begin{eqnarray}\label{MuSimple}
	\nonumber
	\frac{1}{2t}  \left[
	\chi_m(1)+  \chi_m\left(g_0^t\right)(-1)^l +T\left(\chi_m\left(g_0^2\right)\zeta_{2t}^{-2l}\right) \right. \hspace{5cm}& & \\
	\left. + \sum_{i \in \Odd} \left( \pa_1\left(g_{0}^{i}\right)T\left(\chi_m\left(g_0^{i}\right)\zeta_{2t}^{-l}\right) +
	\pa_1\left(g_{0}^{t-i}\right) T\left(\chi_m\left(g_0^{t-i}\right)\zeta_{2t}^{-l}\right)\right) \right] & \in & \Z_{\ge 0}.
	\end{eqnarray}
	We will use this with $l\in \{1,\dots, t-1\}$ and $m\in \{2,4\}$. To facilitate the calculations we collect the following equalities which are all direct application of Lemma \ref{lema Margolis}:
	$$\matriz{{cc}
		\chi_2(1)=3, \quad \chi_2\left(g_0^t\right)=-1 & \chi_4(1)=5,
		\quad \chi_4\left(g_0^t\right)=1, \\ T\left(\chi_2\left(g_0^{2}\right)\zeta_{2t}^{-2l}\right) = t w_l -3, & T\left(\chi_4\left(g_0^{2}\right)\zeta_{2t}^{-2l}\right) = t W_l -5.
	}$$	
	Moreover, if $i\in \Odd$ then
	$$\matriz{{ll}
		T\left(\chi_2\left(g_0^{i}\right)\zeta_{2t}^{-l}\right) =
		\begin{cases}
		(1-t)(-1)^l, &\text{if } i=i_l;\\
		(-1)^l, & \text{otherwise};
		\end{cases}
		& 	
		T\left(\chi_4\left(g_0^i\right)\zeta_{2t}^{-l}\right) =
		\begin{cases}
		(t+1)(-1)^{l+1}, &\text{if } i=i_l;\\
		(t-1)(-1)^l, & \text{if } i=i_{nl};\\
		(-1)^{l+1}, & \text{otherwise};
		\end{cases}}$$
	$$\matriz{{ll}
		T\left(\chi_2\left(g_0^{t-i}\right)\zeta_{2t}^{-l}\right)=
		\begin{cases}
		(t-3)(-1)^l, & \text{if } i=i_l;\\
		(-3)(-1)^l, & \text{otherwise};
		\end{cases}	
		&
		T\left(\chi_4\left(g_0^{t-i}\right)\zeta_{2t}^{-l}\right)=
		\begin{cases}
		(t-5)(-1)^l, & \text{if } i\in \{i_l,i_{nl}\};\\
		(-5)(-1)^l, & \text{oherwise}.	
		\end{cases}
	}$$
	Plugging this information in \eqref{MuSimple} for $m=2$ and $m=4$ and using (\ref{ValuesAccumulated}),  we obtain
	\begin{equation}\label{MuSimpleChi_2}
	\frac{1}{2}\left( (-1)^l\left( \pa_1\left(g_0^{t-i_l}\right)-\pa_1\left(g_0^{i_l}\right) \right) + w_l \right)\; \in \; \Z_{\ge 0}
	\end{equation}
	and
	\begin{equation}\label{MuSimpleChi_4}
	\frac{1}{2}\left( (-1)^l\left( \pa_1\left(g_0^{i_{nl}}\right) + \pa_1\left(g_0^{t-i_{nl}}\right)
	-\pa_1\left(g_0^{i_l}\right) +\pa_1\left(g_0^{t-i_l}\right)  \right) + W_l \right) \; \in \; \Z_{\ge 0}.
	\end{equation}
	Using \eqref{MuSimpleChi_2} with $l$ and with $t-l$ we obtain
	$\left| \pa_1(g_0^l) - \pa_1(g_0^{t-l})\right| \leq w_l$
	if $l\in \Odd$.
	In particular,
	\begin{equation}\label{pa1NoVarrho}
	\pa_1(g_0^l) = \pa_1(g_0^{t-l}), \text{ if } l\in \Odd\setminus \{1\}.
	\end{equation}
	This together with (\ref{ValuesAccumulated}) yields
	\begin{equation}\label{pa1Varrho}
	\pa_1(g_0) = 1-\sum_{l\in \Odd\setminus\{1\}} \pa_1(g_0^l)=
	1-\sum_{l\in \Odd\setminus\{1\}} \pa_1(g_0^{t-l})=1+\pa_1(g_0^{t-1}).
	\end{equation}
	We now combine \eqref{pa1NoVarrho} and \eqref{pa1Varrho} with \eqref{MuSimpleChi_4} for $l$ and $t-l$.
	When we take $l=1$ we obtain
	$$\pa_1(g_0^{i_n})=\pa_1(g_0^{t-i_n})\in \{0,1\};$$
	and when we take $l\in \Odd\setminus \{1,t-2\}$ we have
	$$\pa_1(g_0^{i_{nl}})=\pa_1(g_0^{t-i_{nl}})=0, \quad \text{if } l\in \Odd\setminus \{1,t-2\}.$$
	As $l\mapsto i_{nl}$ defines a bijection $\Odd\rightarrow \Odd$ mapping $t-2$ to $1$, the latter is equivalent to
	$$\pa_1(g_0^l)=\pa_1(g_0^{t-l})=0, \quad \text{for all } l\in \Odd\setminus\{i_n,1\}.$$
	Thus
	$$\pa_1(g_0)=1+\pa_1(g_0^{t-1})=1-\sum_{l\in \Odd\setminus \{1\}} \pa_1(g_0^{t-l})=1-\pa_1(g_0^{t-i_n})=1-\pa_1(g_0^{i_n})\in \{0,1\}.$$
	Since $\{i_n,t-i_n\}=\{n,n+1\}$, we conclude that either $\pa_{1}(g_0)=1$ and $\pa_{1}(g_0^l)=0$ for every $2\le l \le t-1$, or $\pa_1(g_0^{t-1})=-1$, $\pa_1(g_0^n)=\pa_1(g_0^{n+1})=1$ and $\pa_1(g_0^l)=0$ for every integer $l$ in $[1,t-1]\setminus \{n,n+1,t-1\}$.
	In the first case $\pa_1=\varepsilon^*_1$ and in the latter case $\pa_1=\pa^{(g_0)}_1$, as desired.
	This finishes the necessary part of the proof.
	
	To finish the proof of Theorem~\ref{main} it remains to prove that if $g_0$ is an element of $G$ of order $2t$ then $\pa^{(g_0)}\in \VPA_{2t}(G)$.
	That $\pa^{(g_0)}$ satisfies conditions (V1), (V2) and (V3) follows by straightforward arguments.
	So it remains to show that the following is a non-negative integer for every ordinary or Brauer character:
	$$\mu\left(\zeta_{2t}^l, \pa^{(g_0)}, \chi\right)=\frac{1}{2t}  \left[
	\chi(1)+  \chi\left(g_0^t\right)(-1)^l +T\left(\chi\left(g_0^2\right)\zeta_{2t}^{-2l} + \left(\chi\left(g_0^n\right)
	+ \chi\left(g_0^{n+1}\right) - \chi\left(g_0^{t-1}\right)\right)\zeta_{2t}^{-l}\right) \right].$$
	Actually, it suffices to consider Brauer characters module the prime $p$ dividing $q$.
	This is a consequence of the following remark, which was brought to our attention by Leo Margolis.
	
	\begin{remark}\label{ReductionBrauerp}
		Let $G=\PSL(2,q)$, with $q$ a power of a prime $p$ and let $n$ be an integer coprime to $p$.
		Let $\pa=(\pa_d)_{d\mid n}$ be a list of class functions of $G$ satisfying conditions (V1), (V2) and (V3).
		Then $\pa\in \VPA_n(G)$ if and only if it satisfies (V4) for every irreducible Brauer character of $G$ modulo $p$.
	\end{remark}
	
	\begin{proof}
		Suppose that $\pa\in \VPA_n(G)$ satisfies (V4) for the irreducible Brauer characters of $G$ modulo $p$.
		Observe that $\mu(\zeta_n^l,\pa,\chi)$ is $\Z$-linear in the last argument.
		This implies that if condition (V4) holds for the class functions $f_1,\dots,f_k$
		then it also holds for each linear combination $\chi=a_1f_1+\dots +a_kf_k$ with $a_1,\dots,a_k$ non-negative integers.
		This is also valid for class functions defined on the $t$-regular elements of $G$ for a given prime $t$.
		Therefore, to verify (V4) it is enough to consider irreducible ordinary characters and irreducible Brauer characters of $G$.
		Moreover, as $n$ is coprime with $p$, by (V3), each non-zero summand in
		the expression of $\mu(\zeta_n^l,\pa,\chi)$ correspond to $p$-regular elements, for any ordinary or Brauer character of $G$.
		Thus we only have to consider the restriction of each ordinary character of $G$ to the $p$-regular elements
		and the restriction of Brauer characters of $G$ modulo a prime $t$ to the $\{p,t\}$-regular elements.
		If $t$ is a prime integer then the restriction to the $t$-regular elements of $G$ of an
		irreducible ordinary character of $G$ is a linear combination of the irreducible Brauer characters of $G$ modulo $t$
		with coefficients in the decomposition matrix relative to $t$ and these coefficients are non-negative
		(see e.g. \cite[Section~15.2]{Serre1977}).
		Using the expression of the ordinary characters of $G$ on the $p$-regular elements in terms of the Brauer characters of $G$
		modulo $p$ we deduce that (V4) holds for all the ordinary characters of $G$.
		Let now $t$ be a prime different from $p$.
		The decomposition matrix $A$ of $G$ relative to $t$ is described in \cite{Burkhardt}.
		Every non-zero column of $A$ contains an entry equal to $1$ in a row on which all the other entries are $0$.
		This implies that each Brauer character of $G$ modulo $t$ equals the restriction to the $t$-regular elements of an ordinary character of $G$. Hence (V4) also holds for Brauer characters of $G$ modulo $t$.
	\end{proof}
	
	By Lemma~\ref{LemmaCharacters}.\eqref{ChisIrreducible} and Remark~\ref{ReductionBrauerp}, it remains to show that if $r_0,r_1,\dots,r_k$ are non-negative integers with $r_0+\dots+r_k$ even and $l$ is an integer then $\mu(\zeta_{2t}^l,\pa^{(g_0)},\chi_{r_0,r_1,\dots,r_k})\; \in \; \Z_{\ge 0}$. For that we use that the map $\chi\mapsto \mu(\zeta_{2t}^l,\pa^{(g_0)},\chi)$ is linear and the expression of $\chi_{r_0,r_1,\dots,r_k}$ obtained in Lemma~\ref{LemmaCharacters}.(\ref{ChisInPhisPsis}) in terms of the ordinary characters $1_G$, $\phi_h$ and $\psi_h$ for $h\in \{1,\dots,t\}$. Hence we start considering the latter characters.
	
	In the remainder of the proof $l$ is an integer, $h\in \{1,\dots,t\}$ and $\epsilon=\pm 1$ with $q\equiv \epsilon \bmod 4t$.
	We will use Lemma~\ref{lema Margolis} without specific mention.
	An easy calculation shows that
	\begin{equation}\label{mu01}
	\mu(\zeta_{2t}^l,\pa^{(g_0)},\psi_h)=\frac{q-\epsilon}{2t}
	\quad \text{and} \quad 
	\mu(\zeta_{2t}^{l},\pa^{(g_0)},1_{G})= \begin{cases} 1, & \text{if } 2t\mid l; \\ 0, & \text{otherwise}.\end{cases}
	\end{equation}

	Now we calculate $\mu(\zeta_{2t}^{l},\pa^{(g_0)},\phi_{h})$.
	For that we introduce the following notation:
	$$
	\vartheta_{h,l}=\begin{cases}
	1, & \text{if } h \equiv \pm 2l \bmod 2t, \text{ or }  h \equiv \pm l \bmod t \text{ and } 2\nmid l; \\
	0, & \text{otherwise};
	\end{cases}
	\quad
	\gamma_{i,l}=\begin{cases}
	t, & \text{if }hi\equiv\pm l\bmod t;\\
	0, & \text{if }hi\not\equiv\pm l\bmod t;
	\end{cases}
	$$
	for each $1\leq i\leq t-1$.
	Clearly, we have $\phi_h(g_0^t)=2\epsilon(-1)^h$, $T(\phi_{h}(g_{0}^{i})\zeta_{2t}^{-l}) = (\gamma_{i,l} - 2)\epsilon(-1)^{hi+l}$, $\gamma_{n,l}=\gamma_{n+1,l}$,  $\gamma_{2,2l}=\gamma_{t-1,l}=\gamma_{1,l}$ and
	$$\gamma_{1,l}(1-(-1)^l)+\gamma_{n,l}(-1)^{hn+l}(1+(-1)^h)=2t\vartheta_{h,l}.$$
	Therefore
	$$T\left(\phi_h\left(g_0^2\right)\zeta_{2t}^{-2l} - \phi_h\left(g_0^{t-1}\right)\zeta_{2t}^{-l}\right)=
	\epsilon(\gamma_{1,l} -2)(1-(-1)^l)$$
	and
	$$T\left(\left(\phi_h\left(g_0^n\right) + \phi_h\left(g_0^{n+1}\right)\right)\zeta_{2t}^{-l}\right)=
	\epsilon(\gamma_{n,l}-2)(-1)^{hn+l}(1+(-1)^h))= \begin{cases}
	2\epsilon(\gamma_{n,l}-2)(-1)^{l}, & \text{if } 2\mid h;\\
	0, & \text{if } 2\nmid h.
	\end{cases}$$
	Hence
	\begin{equation}\label{muphi}
	\mu\left(\zeta_{2t}^l, \pa^{(g_0)}, \phi_h\right)=
	\frac{q-\epsilon+\epsilon(\gamma_{1,l}(1-(-1)^l)+\gamma_{n,l}(-1)^{hn+l}(1+(-1)^h))}{2t}=
	\frac{q-\epsilon}{2t}  +\epsilon \; \vartheta_{h,l}.
	\end{equation}
	
	Let $R=(r_0,r_1,\dots,r_k)\in\Z^{k+1}$ with $r_0+r_1+\dots+r_k$ even, and let $2n_h$ be the cardinality of the set $V_{R;h}$ defined in (\ref{Vrh}). Observe that $\mu(\xi,\pa,\chi)$ is linear in the third argument.
	Then, using Lemma~\ref{LemmaCharacters}.(\ref{ChisInPhisPsis}), \eqref{mu01} and \eqref{muphi} we obtain
	\begin{equation}\label{muchivarios}
	\mu(\zeta_{2t}^l,\pa^{(g_0)},\chi_{r_0,\dots,r_{k}}) =
	k_0\mu(\zeta_{2t}^l,\pa^{(g_0)},1_G)+\sum_{h=1}^{t} n_{h}\vartheta_{h,l},
	\end{equation}
	where $k_0$ is either $2n_0$ or $1+2n_0$. This finishes the proof of Theorem~\ref{main}, as the expression in \eqref{muchivarios} is a non-negative integer because $\mu(\zeta_{2t}^l,\pa^{(g_0)},1_G)$ is either $0$ or $1$ by \eqref{mu01}, and all the $n_h$ are non-negative integers.
	
	\textbf{Acknowledgment}. We are very thankful to Andreas B\"{a}chle and Leo Margolis for several useful conversations on Zassenhaus Conjecture.

	\bibliographystyle{alpha}
	\bibliography{ReferencesMSC}
	
\end{document}